\documentclass[11pt]{amsart}
\usepackage{latexsym}
\usepackage{enumerate}
\makeatletter
\usepackage[notintoc]{nomencl}
\makenomenclature
\@namedef{subjclassname@2010}{

\textup{2010} Mathematics Subject Classification}
\usepackage{amssymb, amsmath,amsthm, amsfonts}
\usepackage[english]{babel}
\usepackage[utf8]{inputenc}
\usepackage{comment}
\usepackage{color}
\usepackage{mathrsfs}
\usepackage[pdftex,citecolor=green,linkcolor=red]{hyperref}

\newtheorem{innercustomgeneric}{\customgenericname}
\providecommand{\customgenericname}{}
\newcommand{\newcustomtheorem}[2]{%
  \newenvironment{#1}[1]
  {%
   \renewcommand\customgenericname{#2}%
   \renewcommand\theinnercustomgeneric{##1}%
   \innercustomgeneric
  }
  {\endinnercustomgeneric}
}

\newcustomtheorem{customthm}{Theorem}
\newtheorem{theorem}{Theorem}[section]
\newtheorem{prop}[theorem]{Proposition}

\newtheorem{prob}[theorem]{Problem}
\newtheorem{cor}[theorem]{Corollary}
\newtheorem{lemma}[theorem]{Lemma}
\newtheorem{lem}[theorem]{Lemma}
\theoremstyle{definition}
\newtheorem{remark}{Remark}[section]
\newtheorem{def1}{Definition}[section]

\newcommand{\ra}{\rightarrow}
\newcommand{\bk}{\backslash}
\newcommand{\mc}{\mathcal}
\newcommand{\mb}{\mathbb}
\newcommand{\sg}{\sigma}

\newcommand{\mbf}{\boldsymbol}

\newcommand{\asum}{\sideset{}{^{\ast}}\sum}

\renewcommand{\bar}{\overline}

\renewcommand{\Re}{\textnormal{Re}}

\newcommand{\cond}{\textnormal{cond}}
\renewcommand{\pmod}[1]{\ (\mathrm{mod}\ #1)}

\begin{document}
\title[Multiplicative functions that are close to their mean]{Multiplicative functions that are close to their mean}
\author
{Oleksiy Klurman}
\address{Max Planck Institute for Mathematics, Bonn, Germany and School of Mathematics, University of Bristol, UK}
\email{lklurman@gmail.com}

\author{Alexander P. Mangerel}
\address{Centre de Recherches Math\'ematiques, Universit\'e de Montr\'eal, Montr\'eal, Qu\'ebec, Canada}
\email{smangerel@gmail.com}

\author
{Cosmin Pohoata}
\address{Department of Mathematics, California Institute of Technology, Pasadena,
California, USA}
\email{apohoata@caltech.edu}

\author{Joni Ter\"{a}v\"{a}inen}
\address{Mathematical Institute, University of Oxford, Oxford, UK}
\email{joni.teravainen@maths.ox.ac.uk}

\begin{abstract} 
We introduce a simple sieve-theoretic approach to studying partial sums of multiplicative functions which are close to their mean value. This enables us to obtain various new results as well as strengthen existing results with new proofs.

As a first application, we show that
for a  multiplicative function $f : \mb{N} \ra \{-1,1\},$
\begin{align*}
\limsup_{x\to\infty}\Big|\sum_{n\leq x}\mu^2(n)f(n)\Big|=\infty.   
\end{align*}
This confirms a conjecture of Aymone concerning the discrepancy of square-free supported multiplicative functions.

Secondly, we show that a completely multiplicative function $f : \mb{N} \ra \mb{C}$ satisfies
\begin{align*}
\sum_{n\leq x}f(n)=cx+O(1)    
\end{align*}
with $c\neq 0$ if and only if $f(p)=1$ for all but finitely many primes and $|f(p)|<1$ for the remaining primes. This answers a question of Ruzsa.

For the case $c = 0,$ we show, under the additional hypothesis $$\sum_{p }\frac{1-|f(p)|}{p}  < \infty,$$ that $f$ has bounded partial sums if and only if $f(p) = \chi(p)p^{it}$ for some non-principal Dirichlet character $\chi$ modulo $q$ and $t \in \mb{R}$ except on a finite set of primes that contains the primes dividing $q$, wherein $|f(p)| < 1.$ This provides progress on another problem of Ruzsa and gives a new and simpler proof of a stronger form of Chudakov's conjecture.

Along the way we obtain quantitative bounds for the discrepancy of the {\it modified characters} improving on the previous work of Borwein, Choi and Coons.
\end{abstract}

\maketitle

\section{Introduction}
Let $S^1$ denote the unit circle of the complex plane. The characterization of the behavior of partial sums of multiplicative functions is an active object of study in analytic number theory, being an example of a problem in which the multiplicative structure of an additively structured set, namely an interval, is investigated. Considerations involving the extremal behavior of such partial sums, for example characterizing when these are bounded, was of key importance in Tao's solution in \cite{tao_edp} to the Erd\H{o}s Discrepancy Problem, according to which
$$
\sup_{d,N \geq 1} \Big|\sum_{n \leq N} f(dn)\Big| = \infty
$$
for any sequence $f: \mb{N} \ra S^1$. It is not difficult to see that the completely multiplicative functions, for which the sum is insensitive to dilations by $d$, ought to be extremal for this problem, and indeed Tao's proof proceeds by reducing the problem to the case of completely multiplicative functions and then proving the claim in this case. One of the main consequences of \cite{tao_edp} is that
\begin{equation}\label{partial_sum}
\limsup_{x\to\infty}\Big|\sum_{n\le x}f(n)\Big|=\infty
\end{equation}
for any completely multiplicative function $f:\mathbb{N}\to S^1.$

The objective of the present paper is to introduce a new tool for analyzing partial sums of multiplicative functions that allows us to solve several conjectures and to considerably strengthen some of the earlier results with new and arguably more elegant proofs.

We study in particular questions on the characterization of completely multiplicative functions that are \emph{close to their mean}, in the sense that
\begin{equation} \label{eq_hyp}
\sum_{n \leq x} f(n) = cx + O(1)
\end{equation}
for some $c \in \mb{C}.$

Our first main result concerns the case $c=0$ and a problem treated by Aymone in \cite{Ayo} on the Erd\H{o}s discrepancy problem for square-free supported multiplicative sequences. After Tao's proof of~\eqref{partial_sum}, it is natural to ask whether the same conclusion holds when the sum is restricted to certain subsets of the integers. In~\cite{Ayo}, it is proven that if for a multiplicative function $g:\mathbb{N}\to\{-1,+1\}$ we have
\begin{align}\label{eq31}
 \sum_{n\le x}g(n)\mu^2(n)=O(1),
\end{align}
then there exists a real primitive Dirichlet character $\chi$ of conductor $q$ such that $(q,6)=1$ and $g(2)\chi(2)=g(3)\chi(3)=-1$, and moreover,
\begin{align*}\label{eq8}
\sum_{p}\frac{1-g(p)\chi(p)}{p}<\infty.
\end{align*}
Aymone conjectured in \cite[Conjecture 1.1]{Ayo} that no functions $g$ satisfying \eqref{eq31} exist. Using our new approach, we are able to settle his conjecture.
\begin{theorem}[Square-free Erd\H{o}s  discrepancy problem]\label{thm_sfedp}
Let $g:\mathbb{N}\to\{-1,+1\}$ be a multiplicative function. Then
\[ \limsup_{x \ra \infty} \Big|\sum_{n\le x}g(n)\mu^2(n)\Big| = \infty. \]
\end{theorem}

As another application, we answer a question of Ruzsa \cite{ruzsa} on characterizing completely multiplicative functions that are close to their mean in the sense of \eqref{eq_hyp}.

In the case $c\ne 0,$ Ruzsa posed the following problem (see \cite[Problem 21]{ruzsa}).
\begin{prob}[Ruzsa] \label{ques_ruzsa1}
Let $f:\mathbb{N}\to \mathbb{C}$ be a completely multiplicative function that satisfies 
\begin{align*}
\sum_{n\leq x}f(n)=cx+O(1)   
\end{align*}
as $x\to \infty$ for some $c\neq 0.$ Is it true that for every prime $p$ either $|f(p)|<1$ or $f(p)=1$ holds?
\end{prob}
Our next result confirms this, in fact in a stronger form that gives a complete characterization of completely multiplicative functions that are close to their mean, provided that the mean is non-zero.

\begin{theorem}\label{theo_ruzsa}
A completely multiplicative function $f:\mathbb{N}\to \mathbb{C}$ satisfies 
\begin{align}\label{eq1}
\sum_{n\leq x}f(n)=cx+O(1)    
\end{align}
as $x\to \infty$ for some $c\neq 0$ if and only if there exists a finite set $S$ of primes such that $|f(p)|<1$ for $p\in S$ and $f(p)=1$ for $p\notin S$.
\end{theorem}
In the same paper, Ruzsa also asked for a plausible characterization of completely multiplicative functions with bounded partial sums, i.e., the case $c = 0$ of Problem \ref{ques_ruzsa1} (see Problem 22 of \cite{ruzsa}). We make progress in this direction by finding a characterization of such completely multiplicative functions when we restrict the set where $|f(p)|$ is small.

\begin{theorem}\label{thm_bddsums}
Let $f : \mb{N} \ra \mb{C}$ be a completely multiplicative function such that 
\begin{align}\label{eq_thin}
\sum_p\frac{1-|f(p)|}{p}<\infty.    
\end{align}
Then $f$ satisfies
\begin{equation}\label{partial_sum2}
\limsup_{x \ra \infty} \Big|\sum_{n \leq x} f(n)\Big| < \infty
\end{equation}
if and only if there is a Dirichlet character $\chi$, a real number $t$ and a finite set $S$ of primes such that $f(p) = \chi(p)p^{it}$ if $p \notin S$, and $|f(p)| < 1$ if $p \in S.$ 
\end{theorem}

One can state our condition \eqref{eq_thin} in the equivalent form
\begin{align}\label{equ11}
\sum_{n\leq x}|f(n)|=\Omega(x);    
\end{align}
see the proof of Lemma \ref{le_pret} for this equivalence. 

\begin{remark}
In Section \ref{sec_exs}, we give a series of examples that show that the assumption \eqref{eq_thin} on $|f(p)|$ in Theorem \ref{thm_bddsums} is necessary in the sense that without it there are uncountably many functions $f$ (in fact, cardinality $\mathfrak{c}=|\mathbb{R}|$, even when restricted to functions taking values in $\mathbb{Q}\cap [-1,1]$) violating the statement. We also observe in that section that relaxing the assumption \eqref{eq_thin} (or \eqref{equ11}) on $f$ even slightly would essentially require obtaining a generalization of Tao's two-point Elliott conjecture \cite{tao} to ``sparse'' multiplicative functions (and also with uniformity in the shifts), which to date has not been achieved. It therefore seems that Theorem~\ref{thm_bddsums} is the strongest result that can be obtained on Ruzsa's problem \cite[Problem 22]{ruzsa} without making progress on challenging conjectures on correlations of multiplicative functions.
\end{remark}

Theorems \ref{theo_ruzsa} and \ref{thm_bddsums} allow us to deduce a refinement of Chudakov's conjecture \cite{chu}, \cite{chu2} proven by the first two authors in \cite{klu_man} (the case $c \neq 0$ was treated earlier by Glazkov \cite{gla}). This proof in particular avoids the use of correlation formulas for pretentious multiplicative functions and the delicate analysis of Euler products that results from using them in~\cite{klu_man}.

\begin{cor}[Refinement of Chudakov's Conjecture] \label{cor_glaz}
Let $f: \mb{N} \ra \mb{C}$ be a completely multiplicative function satisfying \eqref{eq1}. \\
a) If $c \neq 0$ and $|f(n)|\in \{0,1\}$ for all $n$, then there is an integer $q \in \mb{N}$ such that $f$ is the principal Dirichlet character modulo $q$. \\
b) If $c = 0$, and $|f(n)|\in \{0,1\}$ for all $n$, and additionally if $S=\{p:\,\, f(p)=0\}$ satisfies $\sum_{p\in S}1/p<\infty$, then there exist $t \in \mb{R}$ and $q \in \mb{N}$ such that  $f(n) = \chi(n)n^{it}$ for some Dirichlet character $\chi$ modulo $q$.
\end{cor}
We remark that previously the works on the subject of classification of multiplicative functions \cite{tao_edp},\cite{Ayo},\cite{dkm},\cite{klu},\cite{klu_man},\cite{Orbits},\cite{dkm1} crucially used much stricter assumptions on the range of $f$ for all $n,$ that is, the range of $f$ was assumed to be finite and also importantly it was required that $f(p)=0$ for only finitely many primes.

Another consequence of our approach is a quantitative statement about the discrepancy of a class of multiplicative functions. We call functions in this class \emph{modified characters} (corresponding to the twisted character $\chi(n)n^{it}$), and they are defined as completely multiplicative functions $f:\mathbb{N}\to \mathbb{U}$ such that $f(p)=\chi(p)p^{it}$ for some Dirichlet character $\chi$ and for all but finitely many primes $p$. Here, $\mb{U}$ denotes the closed unit disc in $\mb{C}$.

As discussed in~\cite{tao_edp}, modified characters (taking values in $\{-1,+1\}$) are believed to be ``extremal" in the Erd\H{o}s discrepancy problem. In particular, if $f$ is a $\pm 1$ valued modified character that differs from a character $\chi$ at only one prime, then a result of Borwein, Choi and Coons \cite{borwein-choi-coons} shows that the maximal partial sums $\sup_{y\leq x}|\sum_{n\leq y}f(n)|$ are $\asymp \log x$, and this is believed to be the lowest possible growth rate of the partial sums for any completely multiplicative $f:\mathbb{N}\to \{-1,+1\}$. Our next corollary confirms the same lower bound for the growth rate for all modified characters.

\begin{cor}
Let $f:\mathbb{N}\to\mathbb{U}$ be a modified character corresponding to $\chi(n)n^{it}$ such that for some prime $r$ we have $f(r)\neq \chi(r)r^{it}$ and $|f(r)|=1$. Then
\begin{align}\label{discbound}
\sum_{n\le x} f(n)=\Omega(\log x).
\end{align}
\end{cor}

The result of \cite{borwein-choi-coons} corresponds to the case $|S|=1$. We note that the proof method of \cite{borwein-choi-coons}, which is based on looking at base $p_0$ expansions for the unique prime such that $f(p)\neq \chi(p)$ for $p\neq p_0$, does not work in the case $|S|\geq 2$. Indeed, relatively little is known about representations of integers simultaneously in two different bases, and therefore we needed a different approach. We conjecture that the optimal growth rate for the sum in \eqref{discbound} is $\Omega((\log x)^{|S|})$.

\subsection{Proof ideas}

The proofs of all of the results mentioned above are crucially based on different variants of a new sieve-theoretic approach to analyzing partial sums of multiplicative functions. We call this approach the \emph{rotation trick}; we explain the method and its naming below.

Let us first recall how the work of Tao~\cite{tao_edp} and subsequent works \cite{klu},~\cite{klu_man},~\cite{Orbits} of a subset of the authors analyze partial sums; we then compare this to our new approach. 

All the works mentioned above start by applying Tao's two-point Elliott conjecture from \cite{tao} to reduce to the case of pretentious functions.
The next step in these works is a second moment argument to show that for $f:\mathbb{N}\to S^1$ multiplicative,
\begin{align}
\label{eq13}
\frac{1}{x}\sum_{n\le x}\Big|\sum_{n\le m\le n+H}f(m)\Big|^2
\end{align}
is large for appropriately chosen $H$ unless $f$ is of a very special form (for example, a modified character). In the case of \cite{tao_edp}, this reduction is achieved with a technical Dirichlet series argument that requires $f$ not to take the value $0$. In the case of~\cite{klu_man} that deals more generally with multiplicative functions that can take finitely many zero values, the approach is based on expanding out the mean square and using formulas for correlations of pretentious functions from~\cite{klu} and a delicate analysis of the arising Euler product factors. 

This second moment approach does not work for our main results in the pretentious case, and so we avoid it by appealing to the rotation trick. What the rotation trick enables us to do is to reduce the case of arbitrary pretentious multiplicative functions to the more manageable case of modified characters. The remaining case of modified characters is not trivial though, and we solve it as follows:
\begin{itemize}
 
\item In the case of Theorem \ref{thm_bddsums}, we exploit some recursive properties of the partial sums in terms of the digits of $x$ in a suitable base, and then appeal to the rotation trick again; see Lemmas \ref{zero_sum} and \ref{zero_sumfinal}.   

\item  In the case of Theorem \ref{thm_sfedp}, we use a soft Dirichlet series argument and zero density estimates that involve bounding the number of common zeros of $L(s,\chi)$ and $\zeta(2s)$; see Proposition \ref{prop_dirichlet} for the details.

\end{itemize}

Let us outline how our rotation trick works in the proof of Theorem \ref{thm_bddsums} to reduce matters from the pretentious case to the modified character case. 

Assuming for the sake of contradiction that there exist infinitely many primes $p$ for which $f(p)\ne \chi(p)p^{it}$ (and assuming $t=0$ for simplicity), we use these primes via the lower bound sieve and the Chinese remainder theorem to construct two short intervals $I$ and $I'$ of length $H$ such that
\begin{align}\label{eq14}
\Big|\sum_{n\in I}f(n)-\sum_{n\in I'}f(n)\Big| \end{align}
is large; this in particular implies that $f$ must exhibit large partial sums, which leads to a contradiction with the assumption of $f$ having bounded partial sums.

Here we describe a rather simplified version of the rotation trick for the function $f:\mathbb{N}\to\mathbb{U}$ in the case that the set $S:=\{p:\,\, f(p)\ne \chi(p)\}$ is thin\footnote{We say that a set of primes $S$ is thin if $\sum_{p \in S} 1/p < \infty$.}. Note that $S$ is a zero-dimensional set from the point of view of sieve theory, and hence sieving works extremely well for this set. We can thus find many integers $m$ for which all the numbers in $[(H!)^2m, (H!)^2m+H]$ have no prime factors from $S\cap [H+1,\infty)$. Letting $p_1,\ldots, p_l>H$ be distinct primes from $S$ and $r_1,\ldots, r_l\in [1,H]$, $k_1,\ldots, k_l\in \mathbb{N}$ be parameters, we can similarly find many $m'$ such that for all $1\leq j\leq l$ we have $p_j^{k_j}\mid \mid (H!)^2m'+r_j$, and such that for all $1\leq r\leq H$ the condition $p\mid (H!)^2m'+r$, $p\in S\cap [H+1,\infty)$ implies $r=r_j$ and $p=p_j$ for some $j$. Now, it is not difficult to see that for $1\leq r\leq H$, $r\neq r_j$, the numbers $(H!)^2m+r$ and $(H!)^2m'+r$ have the same prime factors from $S$ with the same multiplicities, and thus
\begin{align*}
f((H!)^2m+r)=f((H!)^2m'+r),   \quad \textnormal{when}\quad r\neq r_j ,
\end{align*}
since for $p\not \in S$ we already have $f(p)=\chi(p).$ We similarly see that for $r=r_j$ the numbers $(H!)^2m+r$ and $(H!)^2m'+r$ have the same prime factors from $S$ with the same multiplicities, apart from the prime $p_j$ that only divides the latter number and divides it to power $k_j$. Thus
\begin{align*}
f((H!)^2m+r)=(f(p_j)\bar{\chi}(p))^{k_j}f((H!)^2m'+r),\quad \textnormal{when}\quad r=r_j.
\end{align*}
Now, taking $I=[(H!)^2m,(H!)^2+H]$, $I'=[(H!)^2m', (H!)^2m'+H]$, \eqref{eq14} becomes
\begin{align}\label{eq17}
\Big|\sum_{1\leq j\leq l}(1-(f(p)\bar{\chi}(p))^{k_j})f((H!)^2m+r_j)\Big|,    
\end{align}
and this can easily be made large by choosing large $k_j$ and $r_j$ appropriately so that the corresponding values of $f$ and $1-(\chi(p)f(p_j))^{k_j}$ all point in approximately the same direction and are bounded away from zero (using $f(p_j)\neq \chi(p_j)$). We call the underlying idea a rotation trick, since effectively we have used a small set of primes to rotate the terms in the sums over $n\in I$ and $n\in I'$ to point in opposite directions. 

Let us mention that in the general case, where $f$ pretends to be $\chi$ but $\{p:\,\,f(p)\neq \chi(p)\}$ is not thin,  we make use of concentration inequalities (see Lemma~\ref{concentrat} for the details) to guarantee that $f(Qn+a)$ is very ``close" to $\chi(a)$ for many $n\le x$ and appropriately chosen large $Q.$

We note that the proof of Theorem~\ref{theo_ruzsa} does not require the use of Tao's two-point correlation result, but instead reduces to the pretentious case using a version of Hal\'{a}sz's theorem. In the proof of Theorem \ref{thm_sfedp}, we also follow similar ideas as above, but establish the largeness of \eqref{eq17} for suitable $k_j$ and $r_j$ in a slightly different manner. In the case of either theorem, we finish the proof by handling the modified character case as explained above.

{\bf Acknowledgments.} The authors would like to warmly thank Marco Aymone for suggesting that our initial argument could lead to a proof of Theorem~\ref{thm_sfedp}. We are grateful to the referee for a careful reading of the paper and many helpful comments. The fourth author was supported by a Titchmarsh Fellowship from the University of Oxford. Finally, O.K. greatly acknowledges support and excellent working conditions at the Max Planck Institute for Mathematics (Bonn).

\subsection{Notation}\label{subsec_notation} The notation $f(x)=\Omega(g(x))$ stands for $\limsup_{x\to \infty} |f(x)|/|g(x)|>0$.

We denote by $S^1$ and $\mathbb{U}$ the unit circle and the closed unit disc of the complex plane, respectively. By $v_p(n)$ we denote the largest $k$ such that $p^k\mid n$. The notation $(a,b)$ stands for the greatest common divisor of $a$ and $b$, and more generally, if $S\subset \mathbb{N}$ is a set, then $(n,S):=\max_{s\in S}(n,s)$. Finally, we make use of the pretentious distance of Granville and Soundararajan
\begin{align*}
\mathbb{D}(f,g;y,x):=\Big(\sum_{y<p\leq x}\frac{1-\textnormal{Re}(f(p)\overline{g(p)})}{p}\Big)^{1/2},\quad \quad \mathbb{D}(f,g;x):=\mathbb{D}(f,g;1,x).   
\end{align*}

\section{Proof of Theorem \ref{thm_bddsums}}\label{sec: thin}

\subsection{The if direction}

We begin by proving the ``if'' direction of Theorem \ref{thm_bddsums}, which is considerably easier than the ``only if'' direction. We make use of the following lemma.

\begin{lemma}[Small perturbations to $f$ preserve closeness to mean value]\label{le_perturb} Let $f,\tilde{f}:\mathbb{N}\to \mathbb{U}$ be completely multiplicative functions such that $f(p)\neq \tilde{f}(p)$ for only finitely many primes and for those primes we have $|\tilde{f}(p)|<1$. Then for any $c\in \mathbb{C}$ we have
\begin{align}\label{comparison}
\Big|\sum_{n\leq x}\tilde{f}(n)-cP_{f,\tilde{f}}x \Big|  \ll_{f,\tilde{f}}\sup_{y\leq x}\Big|\sum_{n\leq y}f(n)-cy\Big|,   
\end{align}
where $P_{f,\tilde{f}}\neq 0$ is some constant. 
\end{lemma}

\begin{proof} Let $S=\{p:\,\, f(p)\neq \tilde{f}(p)\}=\{p_1,\ldots, p_{|S|}\}$. Let us write $\tilde{f}=f*h$, where $h=\tilde{f}*(\mu f)$. For all $k\geq 1$, we have $h(p^k)=\tilde{f}(p)^k-\tilde{f}(p)^{k-1}f(p)$, so $|h(p^k)|\leq 2C|\tilde{f}(p)|^k$, where $C=1/|\tilde{f}(p)|$ if $\tilde{f}(p)\neq 0$ and $C=0$ if $\tilde{f}(p)=0$. If we denote by $F(x)$ the supremum on the right of \eqref{comparison}, we see that
\begin{align*}
\sum_{n\leq x}\tilde{f}(n)=\sum_{p\mid d\Longrightarrow p\in S}h(d)\sum_{m\leq x/d}f(m)=cx\sum_{p\mid d\Longrightarrow p\in S}\frac{h(d)}{d}+O(F(x)\sum_{p\mid d\Longrightarrow p\in S}|h(d)|).    
\end{align*}
The main term here is $cP_{f,\tilde{f}}x$, and by a short computation of the Euler factors we have
\begin{align*}
P_{f,\tilde{f}}=\prod_{p\in S}\Big(1+\frac{\tilde{f}(p)-f(p)}{p-\tilde{f}(p)}\Big)\neq 0.
\end{align*}
The error term in turn is
\begin{align*}
\ll F(x)\sum_{k_1,\ldots k_{|S|}\geq 0}|f(p_1)|^{k_1}\cdots |f(p_{|S|})|^{k_{|S|}}\ll_{f,\tilde{f}} F(x)
\end{align*}
by the geometric sum formula and the assumption $|\tilde{f}(p)|<1$ for $p\in S$. This completes the proof.
\end{proof}

\begin{proof}[If direction of Theorem \ref{thm_bddsums}] Let $f(p)=\chi(p)p^{it}$ for all $p\not \in S$ and $|f(p)|<1$ for $p\in S$, where $\chi$ is a Dirichlet character modulo $q$. Let $D$ be the product of all the primes in $S$. By Lemma \ref{le_perturb} (with $c=0$), it suffices to consider the case $f(n)=\chi(n)n^{it}1_{(n,D)=1}=\chi'(n)n^{it}$, where $\chi'(n):=\chi(n)1_{(n,D)=1}$ is a non-principal character modulo $Q:=qD$. 

When $t = 0$, the boundedness of the partial sums of $\chi'(n)n^{it}$ follows from the orthogonality of Dirichlet characters, so assume that $t \neq 0$. In this case,
\begin{align}\label{equ1}
\sum_{n \leq x} \chi'(n)n^{it} &= \asum_{a \pmod{Q}} \chi'(a)\sum_{m \leq x/Q} (mQ + a)^{it} + O_{Q}(1)\nonumber \\
&= Q^{it}\asum_{a \pmod{Q}} \chi'(a)\sum_{m \leq x/D'} m^{it}(1+a/(Qm))^{it} + O_Q(1),
\end{align}
where the asterisked sum $\asum_{a \pmod{Q}}$ denotes a sum over residue classes $a$ that are coprime to $Q$. By a Taylor approximation, for $m$ sufficiently large in terms of $|t|$ we have
\begin{align*}
(1+a/(Qm))^{it} = 1 + \frac{ita}{Qm} + O_t\Big(\frac{a^2}{Q^2m^2}\Big),
\end{align*}
and using this and  the orthogonality of characters \eqref{equ1} becomes
\begin{align*}
Q^{it}\asum_{a \pmod{Q}} \chi'(a) \sum_{m \leq x/Q} m^{it}\Big(1+\frac{ita}{Qm}\Big) + O_{Q,t}(1) = \frac{it}{Q^{1-it}} \asum_{a \pmod{Q}} a\chi'(a) \sum_{m \leq x/Q} \frac{1}{m^{1-it}} + O_{Q,t}(1).
\end{align*}
Since $Q = O(1)$, we easily find that the $m$ sum is $\zeta(1+1/\log x -it) + O_Q(1) = O_{Q,t}(1)$ whenever $t \neq 0$. Hence, $\chi'(n)n^{it}$ has bounded partial sums. \end{proof}

\subsection{Reduction to pretentious functions}

In the next two lemmas, we will reduce to the case where $f$ pretends to be $\chi(n)n^{it}$ for some character $\chi$ and $t\neq 0$.

\begin{lemma}\label{le_1bdd}
Let $f: \mb{N} \ra \mb{C}$ be a completely multiplicative function that satisfies $\sum_{n\leq x}f(n)=O(x^{o(1)})$. Then $|f(n)| \leq 1$ for all $n$.
\end{lemma}
\begin{proof} Suppose on the contrary that $|f(p_0)|>1$ for some prime $p_0$. Fix $x \geq p_0$. Then, by the assumption on $f$, for $k=\lfloor (\log x)/(\log p_0)\rfloor$ we have
\begin{align*}
f(p_0^k)=\sum_{n\leq p_0^{k}}f(n)-\sum_{n\leq p_0^k-1}f(n)=O(x^{o(1)}),    
\end{align*}
but on the other hand we have $|f(p_0^k)|=|f(p_0)|^k\gg x^{c}$ for some $c>0$, a contradiction.
\end{proof}

\begin{lemma} \label{le_pret}
Suppose that $f: \mb{N} \ra \mb{U}$ satisfies $\sum_{n\leq x}f(n)=O(1)$, and that \eqref{eq_thin} holds. Then there is a primitive, non-principal Dirichlet character $\chi$ and a real number $t$ such that 
$$
\sum_p \frac{1-\Re(f(p)\bar{\chi}(p)p^{-it})}{p} < \infty.
$$
\end{lemma}

\begin{proof}
We first claim that \eqref{eq_thin} is equivalent to
\begin{align}\label{eq_thin2}
\sum_p \frac{1-|f(p)|^2}{p}<\infty.    
\end{align}
This follows from $1-|f(p)|\leq 1-|f(p)|^2\leq 3(1-|f(p)|)$. Then we claim that \eqref{eq_thin2} is equivalent to
\begin{align}\label{eq_logmean}
\sum_{n \leq x} \frac{|f(n)|^2}{n} \gg \log x.
\end{align}
Indeed, by e.g. \cite[Th\'{e}or\`{e}me 1.1]{tenLB} the left-hand side is $\asymp \prod_{p\leq x}(1+|f(p)|^2/p)$, which is $\asymp \log x$ if and only if \eqref{eq_thin2} holds. 

Now, given $H \geq 1$ and $x$ chosen sufficiently large in terms of $H$, we have
\begin{align}\label{equ10}
1 &\gg \frac{1}{\log x}\sum_{n \leq x} \frac{1}{n}\Big|\sum_{n \leq m \leq n + H} f(n)\Big|^2 \\
&= \frac{1}{\log x} \sum_{|h| \leq H} \sum_{\max\{1,1-h\} \leq m \leq \min\{x,x+h\}} f(m)\bar{f}(m+h) \sum_{\max\{m-H,m+h-H\} \leq n \leq \min\{m,m+h\}} \frac{1}{n}\nonumber\\ 
&\geq \frac{H}{\log x} \sum_{n \leq x} \frac{|f(n)|^2}{n} - \frac{1}{\log x} \sum_{1 \leq |h| \leq H} (H+1-|h|) \Big|\sum_{1 \leq n \leq x} \frac{f(n)\bar{f}(n+h)}{n}\Big| - O(1),\nonumber
\end{align}
and therefore we must have
$$
\max_{1 \leq |h| \leq H} \Big|\sum_{1 \leq n \leq x} \frac{f(n)\bar{f}(n+h)}{n}\Big| \geq \frac{1}{H}\sum_{n \leq x} \frac{|f(n)|^2}{n} - O(1) \gg_H \log x.
$$
By the two-point logarithmic Elliott conjecture \cite[Theorem 1.3]{tao}, there exists a number $A=O_H(1)$ such that $\inf_{\cond(\chi)\leq A}\inf_{|t|\leq Ax}\mathbb{D}(f,\chi(n)n^{it};x)\ll_H 1$. Since this holds for all $x\gg_H 1$, by a lemma of Elliott \cite[Lemma 17]{elliott2010} (see also \cite[Lemma 2.3]{klu_man}) we can find a primitive character $\chi\pmod q$ with $q=O_H(1)$ and a real $t = O_H(1)$, both independent of $x$, such that 
$$
\sum_p \frac{1-\text{Re}(f(p)\bar{\chi}(p)p^{-it})}{p} < \infty.
$$
It remains to check that $\chi$ is non-principal. If $\chi$ were principal then by Delange's theorem (Lemma \ref{le_hal}) we would have
\begin{align*}
 \Big|\sum_{n \leq x} f(n)\Big|&=(1+o(1))\frac{x}{|1+it|} \prod_{p \leq x} \Big(1-\frac{1}{p}\Big) \Big|1-\frac{f(p)p^{-it}}{p}\Big|^{-1}\\
 &\gg_t  x\exp\Big(-\sum_{p \leq x} \frac{1-\Re(f(p)p^{-it})}{p}\Big) \gg x,
\end{align*}
a contradiction with the boundedness of the partial sums when $x$ is sufficiently large. Thus $\chi$ must be non-principal.
\end{proof}

\subsection{Reduction to modified characters}

Our next task is to prove the following; this will be achieved using the  `rotation trick" described in the introduction. 

\begin{prop}\label{all_finite}
Let $f:\mathbb{N}\to\mathbb{U}$ be a completely multiplicative function such that we have $\mathbb{D}(f,\chi(n)n^{it};\infty)<\infty$ for some non-principal Dirichlet character $\chi\pmod q$ and some $t\in\mathbb{R}.$ Suppose that $\sum_{n\leq x}f(n)=O(1)$. Then we have $f(p)=\chi(p)p^{it}$ for all but finitely many primes $p$.
\end{prop}

Before beginning the proof, we need a concentration inequality for multiplicative functions (closely related to the Tur\'an--Kubilius inequality) which we shall repeatedly use in the course of the proof.

Let
\[F(Q):=\sum_{\substack{p\le x\\p\nmid Q}}\frac{f(p)\bar{\chi(p)}p^{-it}-1}{p}.\]
\begin{lemma}[Concentration inequality for multiplicative functions]\label{concentrat}
Let $f:\mathbb{N}\to\mathbb{U}$ be a multiplicative function such that $\mathbb{D}(f,\chi(n)n^{it};\infty)<\infty$ for some Dirichlet character of conductor $q.$ Let $Q,N\geq 1$ be integers with $\prod_{p\leq N}p\mid Q$ and $q\mid Q.$ Then for any $1\leq a\leq Q$ with $(a,Q)=1$ we have 
\begin{align}\label{equ12}\sum_{n\le x}|f(Qn+a)-\chi(a)(Qn)^{it}\exp\Big(F(Q)\Big)|\ll x\Big(\mathbb{D}(1,f\bar{\chi}(n)n^{-it};N,x)+\frac{1}{N^{1/2}}\Big).
\end{align}
\end{lemma}

\begin{proof} This is essentially the Tur\'an--Kubilius inequality adapted to multiplicative functions; see e.g. \cite[Proposition 2.3]{klu} or  \cite[Lemma 2.10]{Orbits}. For the sake of completeness, we sketch the proof. 

Let us write $f(n)=f'(n)\tilde{\chi}(n)n^{it}$ for some multiplicative $f'$ and for $\tilde{\chi}$ the completely multiplicative function given by $\tilde{\chi}(p)=\chi(p)$ if $p\nmid q$ and $\tilde{\chi}(p)=1$ if $p\mid q$. Since $f(Qn+a)=\chi(a)(Qn)^{it}f'(Qn+a) + O_t(1/n)$, it suffices to consider the function $f'$ (upon admitting an additional error term of size $O_t(\log x)$, which is negligible).  Therefore, we may assume that $\chi=1$, $t=0$ in the statement.

Let $h:\mathbb{N}\to \mathbb{C}$ be the additive function given on prime powers by $h(p^k)=f(p^k)-1$. Using $z=e^{z-1}+O(|z-1|^2)$ for $|z|\leq 1$, we have
\begin{align*}
f(Qn+a)=\prod_{\substack{p^k\mid \mid Qn+a\\p>N}}f(p^k)=  \prod_{\substack{p^k\mid \mid Qn+a\\p>N}}(\exp(h(p^k))+O(|h(p^k)|^2)).  
\end{align*}
Using the inequality $|\prod_{i\leq k}z_i-\prod_{i\leq k}w_i|\leq \sum_{i\leq k}|z_i-w_i|$ (which is proven by induction), the above becomes
\begin{align*}
\exp(h(Qn+a))+O\Big(\sum_{\substack{p^k\mid \mid Qn+a\\p>N}}|h(p^k)|^2\Big).     
\end{align*}
Hence, the left-hand side of \eqref{equ12} is 
\begin{align}\label{equ13}
\ll \sum_{n\leq x}|\exp(h(Qn+a))-\exp(F(Q))|+\sum_{n\leq x}\sum_{\substack{p^k\mid \mid Qn+a\\p>N}}|h(p^k)|^2.    
\end{align}
In the first sum in \eqref{equ13}, we may replace $F(Q)$ with $\mu_h:=\sum_{\substack{p^k\leq x \\ p \nmid Q}}h(p^k)/p^k\cdot (1-1/p)$, since $F(Q)=\mu_h+O(1/N)$. Now, applying first $|e^z-e^w|\ll |z-w|$ for $\textnormal{Re}(z),\textnormal{Re}(w)\leq 0$, then Cauchy--Schwarz and then the Tur\'an--Kubilius inequality \cite[Theorem III.3.1]{Ten}, the first sum in \eqref{equ13} becomes \begin{align}\label{equ14}
\ll x\Big(\sum_{p>N,k\geq 1}\frac{|h(p^k)|^2}{p^k}\Big)^{1/2}.   
\end{align}
The contribution of the summands with $k\geq 2$ is trivially $\ll 1/N$. For the other summands, we note that $|h(p)|^2=2(1-\textnormal{Re}(f(p)))$, so we obtain the claimed error term.

For the second sum in \eqref{equ13}, we swap the order of summation to obtain the bound $\ll x\sum_{p>N,k\geq 1}|h(p^k)|^2/p^k$, which again is an admissible contribution. \end{proof}

We are now ready to prove Proposition~\ref{all_finite}.

\begin{proof}[Proof of Proposition \ref{all_finite}]

Let $\varepsilon, H,w,x$ be quantities obeying the hierarchy
\begin{align}\label{equ6}
1\ll \varepsilon^{-1}\ll H\ll w\ll x,    
\end{align}
with each parameter large enough in terms of the ones to the left of it. Also let $W:=\prod_{p\leq w}p^{w}$. 

In the proof of Lemma \ref{le_pret}, we showed that \eqref{eq_thin} is equivalent to \eqref{eq_logmean}, and hence there exists a constant $c>0$ such that for $\gg H$ integers $m$ we have $|f(m)|\geq c$.

By the pigeonhole principle, we can find a subset $r_1,\ldots, r_k$ of $\{m\leq H:\,\, |f(m)|\geq c\}$ of size $k\gg H$ such that for some complex number $\alpha\in (-\pi,\pi]$ we have 
\begin{align}\label{eq_arg}
|\textnormal{arg}(f(r_j)r_j^{-it})-\alpha|\leq 1/100    
\end{align}
for all $j\leq k$. Define the sector
\begin{align*}
A:=\{z\in \mathbb{C}\setminus\{0\}:\,\, \textnormal{arg}(z)\in [\frac{2\pi}{3},\frac{4\pi}{3}]\pmod{2\pi}\}\cup \{0\}.    
\end{align*}

Now we suppose that the set $\mc{P} := \{p:\,\, f(p)\ne \chi(p)p^{it}\}$ is infinite. 

Let
$\psi(n):=\chi(n)n^{it}.$   
By the pigeonhole principle, we choose large prime powers $p_1^{\alpha_1},\dots, p_k^{\alpha_k}$ 
such that the following are simultaneously satisfied:
\begin{enumerate}[(i)]
\item $p_i\in \mathcal{P}$;\\
\item $p_{i+1}>2p_i>2w$ for all $i\leq k-1$;\\
\item $f\bar{\psi}(p_i^{\alpha_i})\in A$ for all $i\leq k.$\\
\end{enumerate}
Property (iii) can indeed be guaranteed, since if $z\neq 0,\, 1$ is any complex number, there exists $j$ such that $z^j\in A$ (if $z/|z|=e(a/b)$ for some $b\geq 2$, then we can easily find $j$ with this property, and if $z/|z|$ is not a root of unity, then the sequence $((z/|z|)^j)_{j\geq 1}$ is dense on the unit circle).

Next, we apply the Chinese remainder theorem to choose a residue class $n_0\pmod{\prod_{i\le k}p_i^{\alpha_i}}$ such that $p_i^{\alpha_i}\vert \frac{W}{r_i}n+1$ whenever $n\equiv n_0\pmod{\prod_{i\leq k}p_i^{\alpha_i}}$. For any $1 \leq j\leq H$ and $n=n_0+m\prod_{i\le k}p^{\alpha_i}$, we observe that 
\begin{align*}f(Wn+r_j)&=f(r_j)f(p_j^{\alpha_j})f\Big(\frac{W}{r_j}\prod_{\substack{1\leq i\leq k\\i\neq j}}p_i^{\alpha_i}m+\frac{\frac{W}{r_j}n_0+1}{p_j^{\alpha_j}}\Big)\\
f(Wn+j)&=f(j)f\Big(\frac{W}{j}\prod_{1\leq i\leq k}p_i^{\alpha_i}m+\frac{W}{j}n_0+1\Big).
\end{align*}
Note also that $((W/r_i)n_0+1)/p_i^{\alpha_i}\equiv p_i^{-\alpha_i}\pmod q.$ We now use Lemma~\ref{concentrat}, the union bound,  and the fact (following from (ii)) that 
\begin{align}\label{equ5}
F(p_1\cdots p_kW)=F(W)+o_{w\to \infty}(1)
\end{align}
to guarantee that a proportion $1-o_{w\to \infty}(1)$ of numbers $(1-\varepsilon)x\leq n\leq x$, $n\equiv n_0\pmod{\prod_{i\le k}p^{\alpha_i}}$ satisfy 
$$|f(Wn+r_j)-f(r_j)f\bar{\psi}(p_j^{\alpha_j})(Wx/r_j)^{it}\exp(F(W))|\ll \varepsilon\quad \textnormal{for all}\quad 1\leq j\leq k.$$
and
$$|f(Wn+j)-f(j)(Wx/j)^{it}\exp(F(W))|\ll \varepsilon\quad \textnormal{for all}\quad 1\leq j\leq H.$$
Thus, there exists $(1-\varepsilon)x\leq n\le x$ such that 
\begin{align}\label{rotate_sum2}\begin{split}
&|\sum_{j\le k}f(Wn+r_j)-(Wx)^{it}\exp(F(W))\sum_{j\le k}f(r_j)f\bar{\psi}(p_j^{\alpha_j})r_j^{-it}|\ll \varepsilon k,\\
&|\sum_{\substack{1\leq j\le H\\j\neq r_1,\ldots, r_k}}f(Wn+j)-(Wx)^{it}\exp(F(W))\sum_{\substack{1\leq j\le H\\j\neq r_1,\ldots, r_k}}f(j)j^{-it}|\ll \varepsilon H,\\
&|\sum_{1\leq j\le H}f(Wn+j)-(Wx)^{it}\exp(F(W))\sum_{1\leq j\leq H}f(j)j^{-it}|\ll \varepsilon H.
\end{split}
\end{align}
Applying the triangle inequality to combine the first two estimates of \eqref{rotate_sum2}, we see that
\begin{align*}
\sum_{j\leq H}f(Wn+j)=(Wx)^{it}\exp(F(W))\Big(\sum_{j\leq H}f(j)j^{-it}+\sum_{j\leq k}(f\bar{\psi}(p_j^{\alpha_j})-1)f(r_j)r_j^{-it}\Big)+O(\varepsilon H).    
\end{align*}
Comparing with the last estimate in \eqref{rotate_sum2}, we obtain 
\begin{align}\label{finalsum}
\Big|\sum_{j\leq k}(f\bar{\psi}(p_j^{\alpha_j})-1)f(r_j)r_j^{-it}\Big|\ll \varepsilon H.
\end{align}
However,  this is impossible since by (iii) and the assumption $|f(r_j)|\geq c$ for all $j$ we have 
$$
|(f\bar{\psi}(p_j^{\alpha_j})-1)f(r_j)r_j^{-it}|\geq c/10,
$$ and by \eqref{eq_arg} and (iii) we have
\begin{align*}
&\textnormal{arg}((f\bar{\psi}(p_j^{\alpha_j})-1)f(r_j)r_j^{-it})\\
&=\textnormal{arg}(f\bar{\psi}(p_j^{\alpha_j})-1)+\textnormal{arg}(f(r_j)r_j^{-it})\\
&\in [\alpha-\pi/3-1/100,\alpha+\pi/3+1/100]\pmod{2\pi}\},
\end{align*}
and so $f\bar{\psi}(p_j^{\alpha_j})-1)f(r_j)r_j^{-it}$ lies in a sector of angle $2\pi/3+1/50<\pi$.
\end{proof}

\subsection{The case of modified characters}

In the remaining case where $f$ is a modified character, we will in fact prove the following quantitative bound on the partial sums.
\begin{prop}\label{generalized_quant}
Let $f:\mathbb{N}\to\mathbb{U}$ be a completely multiplicative function, such that $f(p)=\chi(p)p^{it}$ for all primes $p\not \in S,$ for some finite set of primes $S.$ Suppose further, that there exists $r\in S$ such that $|f(r)|=1$ and $f(r)\ne \chi(r)r^{it}.$ Then
\[\sum_{n\le x} f(n)=\Omega(\log x).\]
\end{prop}

We perform the following simple albeit very useful reduction.
\begin{lemma}[Reducing to $t=0$ and to only one exceptional prime]\label{zero_except1}
In order to prove Proposition \ref{generalized_quant}, it suffices to do so in the case where $f(p)=\chi(p)$ for all but one prime $r$, with $|f(r)|=1$. 
\end{lemma}

\begin{proof}
Suppose $f(p) = \chi(p)p^{it}$ for all but a finite number of primes $p$. Let $Q$ be the product of all primes $p \neq r$ such that $f(p)\neq \chi(p)p^{it}$. By Lemma \ref{le_perturb}, we may replace $f$ by another completely multiplicative function $\tilde{f}$ such that  $\tilde{f}(p) = 0$ for all $p|Q$; for convenience, we continue to denote this new function by $f$.  Replacing $\chi$ by $\chi 1_{(\cdot,Q) = 1}$ (which we will denote by $\chi$ as well for convenience), it suffices to prove the case $f(p)=\chi(p)p^{it}$ for all $p\neq r$ and $|f(r)|=1$, $f(r)\neq \chi(r)r^{it}$, where $\chi \pmod Q$ is a non-principal (though not necessarily primitive) character. 

The case assumed in Lemma \ref{zero_except1} implies that
\begin{align}\label{equ3}
\Big|\sum_{n\leq H}f(n)n^{-it}\Big|\geq c_f \log H    
\end{align}
for infinitely many integers $H\geq 1$. Choose an integer $K$ such that $r^KQ\asymp H^3$ and let $x=r^KQ$. We now have
\begin{align}\label{equ2}
\sum_{x\leq n\leq x+H}f(n)n^{-it}=\sum_{0\leq k\leq K+Q}(f(r)r^{-it})^k\Big(S\Big(\left\lfloor \frac{x+H}{r^k}\right\rfloor\Big)-S\Big(\left\lfloor \frac{x}{r^k}\right\rfloor\Big)\Big),    
\end{align}
where $S(x):=\sum_{\substack{n\leq x \\ (n,r) = 1}}\chi(n)$. Since $S(0)=0$ and $S(x)$ depends only on  $\lfloor x\rfloor\pmod rQ$ and $\lfloor x/r^k \rfloor\equiv 0\pmod rQ$ for $0\leq k\leq K-1$ by our choice of $x$, \eqref{equ2} becomes
\begin{align*}
\sum_{0\leq k\leq K+Q}(f(r)r^{-it})^kS\Big(\left\lfloor \frac{H}{r^k}\right\rfloor\Big)+O_Q(1)=\sum_{n\leq H}f(n)n^{-it}+O_Q(1).    
\end{align*}
Combining this with the Taylor approximation $(x+h)^{-it}=x^{-it}+O_t(H/x)$ for $|h|\leq H$, we have 
\begin{align*}
 x^{-it}\sum_{x\leq n\leq x+H}f(n)=\sum_{n\leq H}f(n)n^{-it}+O(H^2/x)+O(1).   
\end{align*}
Recalling \eqref{equ3} and $H\asymp x^{1/3}$, the claim follows.
\end{proof}

Define the modified character $\chi_{r,z}$, where $r$ is a prime and $z\in S^1$, $z\neq \chi(r)$, by defining on the primes
\begin{align}\label{equ8}
\chi_{r,z}(p)=\begin{cases}\chi(p),\quad p\neq r\\z,\quad p=r.\end{cases}
\end{align}
In view of Lemma \ref{zero_except1}, the proof of Theorem \ref{thm_bddsums} will be complete once we show that
\begin{align}\label{equ4}
\sum_{n\leq x}\chi_{r,z}(n)=\Omega(\log x)    
\end{align}
for any $r$, any $z\in S^1\setminus \{\chi(r)\}$, and any non-principal $\chi \pmod q$. Here we may assume that $\chi(r)\neq 0$, since otherwise $\chi(n)=\chi'(n)1_{(n,r)=1}$, where $\chi'$ is a non-principal character modulo $q'=q/r^{v_r(q)}$ and $\chi'_{r,z}=\chi_{r,z}$.

The proof of \eqref{equ4} will be achieved in the next two lemmas.

\begin{lemma}[Obtaining a zero sum condition]\label{zero_sum}
Let the notation be as above, and suppose that \eqref{equ4} fails. Then for all $m\geq 1$ we have
\[\sum_{n\le mq}\chi_{r,z}(n)=0.\]
\end{lemma}
\begin{proof}
Let $$\Sigma(x):=\sum_{n\le x}\chi_{r,z}(n)\quad \textnormal{and}\quad S(x):=\sum_{\substack{n\le x\\ (n,r)=1}}\chi(n).$$ 
Suppose that the claim is not true. Then there exists $A=m_0q$ such that $\sum_{n\le A}\chi_{r,z}(n)\neq 0.$ 

Fix $K \in \mb{N}$. Let $x,y$ be any positive integers with $x\equiv y\equiv 0\pmod q$ and $y<r^K$. We claim that
\begin{align}\label{iterate}
\Sigma(r^Kx+y)=z^K\Sigma(x)+\Sigma(y). \end{align}
To see this, observe that for any $X\geq 1$ we have
$$\Sigma(X)=\sum_{k\geq 0}z^kS\Big(\left\lfloor \frac{X}{r^k}\right\rfloor\Big),$$
so
$$\Sigma(r^Kx+y)=\sum_{k\geq 0}z^kS\Big(\left\lfloor \frac{r^Kx+y}{r^k}\right\rfloor\Big).$$
Note that for $k<K$ we have
$$\left\lfloor \frac{r^Kx+y}{r^k}\right\rfloor=r^{K-k}x+\left\lfloor \frac{y}{r^k}\right\rfloor\equiv  \left\lfloor \frac{y}{r^k}\right\rfloor \pmod{rq},$$
whereas for $k\geq K$ we have
$$\left\lfloor \frac{r^Kx+y}{r^k}\right\rfloor=\left\lfloor \frac{x}{r^{k-K}}\right\rfloor,$$
since if the floor was one larger, by the assumption $r^{K}>y$ we would have
$$\Big\{\frac{x}{r^{k-K}}\Big\}>1-\frac{y}{r^k}>1-\frac{1}{r^{k-K}},$$
which is not possible. This gives \eqref{iterate}.

In what follows, let $K=O(1)$ be such that $r^{K}>A$. Iterating \eqref{iterate}, for any $0<m_1<m_2<\cdots<m_J$ and $J\geq 1$ we see that
\begin{align}\label{baseexpansion}
\Sigma(r^{m_JK}A+\cdots +r^{m_2K}A+r^{m_1K}A+A)=\Sigma(A)(1+z^{Km_1}+z^{Km_2}+\cdots +z^{Km_J})
\end{align}
Let $c=1/(10K r)$, and choose $\{m_j\}_{j \leq J}$ to consist of those $m\leq c\log x$ that satisfy $|z^{mK}-1|\leq 1/100$; the number $J$ of such $m$ is $\gg_z \log x$ (if $z$ has irrational argument, this follows from the equidistribution of $\{z^{mK}\}_m$, while if $z$ has rational argument this is immediate). Since $r^{m_JK}A+\cdots +r^{m_1K}+A<x$, see wee that there exists $n\leq x$ such that $|\Sigma(n)|\gg \log x$. This gives the desired contradiction.
\end{proof}

We then finish the proof with the following lemma that is based on the rotation trick (in fact, a simpler version than before).
\begin{lemma}[Finishing the proof]\label{zero_sumfinal}
Let the notation be as above. Then \eqref{equ4} holds for any $z\in S^1\setminus \{\chi(r)\}$. 
\end{lemma}

\begin{proof} Assume \eqref{equ4} fails. By Lemma \ref{zero_sum}, we may assume that $\sum_{n\le mq}\chi_{r,z}(n)=0$ for all $m\geq 1$. Also, as mentioned above, we may assume that $\chi(r)\neq 0$, so $(r,q)=1$.

We employ the rotation trick. Choose a large integer $m\ge 10 \log q/\log r$ and some large prime $P\ge 10q$ with $P\equiv r\pmod q.$ We now select two integers $s_1$, $s_2$ such that
\begin{itemize}
    \item $r^ms_1\equiv 1 \pmod q$,
    
    \item $r^{m-1}Ps_2\equiv 1 \pmod q$,
    
    \item $s_1\equiv s_2\equiv 1\pmod{r}$.
\end{itemize}
Let $k_m$ and $\ell_m$ be the unique integers such that $qk_m< r^ms_1\leq  q(k_m+1)$ and $q\ell_{m}< r^{m-1}Ps_2\leq  q(\ell_m+1)$; by choosing $s_1,s_2$ appropriately we can ensure that $k_m<\ell_m$. By Lemma \ref{zero_sum}, we have
\begin{align}\label{zerodifference}
\sum_{qk_m< n\leq q(k_m+1)} \chi_{r,z}(n)-\sum_{q\ell_{m}< n\leq q(\ell_{m}+1)} \chi_{r,z}(n)=0.
\end{align}
Let $L_m = \ell_m - k_m$. Then we can also write the left-hand side of \eqref{zerodifference} as
\begin{align}\label{zerodifference2}
\sum_{qk_m< n\leq  q(k_m+1)} z^{v_r(n)}\chi(n/r^{v_r(n)})-\sum_{qk_m< n\leq q(k_m+1)} z^{v_r(n+qL_m)}\chi((n+qL_m)/r^{v_r(n+qL_m)}).    
\end{align}
We have $qL_m=r^{m-1}Ps_2-r^ms_1\equiv 0\pmod{r^{m-1}}$ by construction, so $v_r(n+qL_m)=v_r(n)\leq \log q/\log r$ for all $n\in (qk_m,q(k_m+1)]$, $n\neq r^ms_1$. Hence, in \eqref{zerodifference2} all the terms cancel out apart from the $n=r^ms_1$ terms, so we are left with
\begin{align*}
z^{v_r(r^{m}s_1)}\chi(r^ms_1/r^{v_r(r^ms_1)})-
z^{v_r(r^{m-1}Ps_2)}\chi(r^{m-1}Ps_2/r^{v_r(r^{m-1}Ps_2)})=0.    
\end{align*}
Since $z\in S^1$, this implies that
\begin{align*}
z\chi(s_1)=\chi(P)\chi(s_2),      
\end{align*}
and since $s_1\equiv s_2\pmod q$,  $(s_i,q)=1$ and $P\equiv r\pmod q$, we must have $z=\chi(r)$, which is the desired contradiction.
\end{proof}

The proof of Theorem \ref{thm_bddsums} is now complete.

\begin{proof}[Proof of Corollary \ref{cor_glaz} b)]
Let $f:\mathbb{N}\to S^1\cup\{0\}$ be as in the corollary. Then, by Theorem~\ref{thm_bddsums}, there is a primitive character $\chi$ modulo $q$ and a $t \in \mb{R}$ such that for all but a finite set $S'$ of primes $p$ we have $f(p) = \chi(p)p^{it}$. For those primes $p \in S'$ we have $|f(p)| < 1$, which implies that $f(p) = 0$. Thus, $f$ is of the form $\chi'(n)n^{it}$, where $\chi'$ is induced by $\chi$.
\end{proof}

\section{Proof of Theorem \ref{theo_ruzsa}}

\subsection{The if direction.} For the proof of Theorem \ref{theo_ruzsa}, we again start with the much easier ``if'' direction. 

\begin{proof}[If direction of Theorem \ref{theo_ruzsa}] By Lemma \ref{le_perturb}, it suffices to show the closeness to mean value in the case $f(n)=1_{(n,Q)=1}$. For this function, we of course have
\begin{align*}
\sum_{n\leq x}1_{(n,Q)=1}=\frac{\varphi(Q)}{Q}x+O(1),    
\end{align*}
so the claim follows.
\end{proof}

\subsection{The only if direction.}

We introduce the following definition. 

\begin{def1} We say that a multiplicative function $f:\mathbb{N}\to \mathbb{C}$ has \emph{property R} if it satisfies \eqref{eq1} for some $c\neq 0$.
\end{def1}

We first state a few lemmas. The first lemma we need states the well-known theorems of Hal\'asz and Delange on mean values of multiplicative functions.

\begin{lemma}[Hal\'asz and Delange theorems]\label{le_hal}
Let $f:\mathbb{N}\to\mathbb{U}$ be a multiplicative function. If the sum
\begin{align*}
\sum_{p}\frac{1-\Re(f(p)p^{-it})}{p}    
\end{align*}
converges for some $t\in \mathbb{R}$, then
\begin{align*}
\frac{1}{x}\sum_{n\leq x}f(n)=\frac{x^{it}}{1+it}\prod_{p\leq x}\Big(1-\frac{1}{p}\Big)\Big(1+\sum_{k\ge 1}\frac{f(p^k)p^{-ikt}}{p^k}\Big)+o(1).    
\end{align*}
Otherwise, we have
\begin{align*}
\frac{1}{x}\sum_{n\leq x}f(n)=o(1).    
\end{align*}
\end{lemma}

\begin{proof}
See \cite[Theorem III.4.5]{Ten}.
\end{proof}

\begin{lemma}\label{le0} Suppose that the completely multiplicative function $f:\mathbb{N}\to \mathbb{C}$ has property R. Then we have $|f(n)|\leq 1$ for all $n$.
\end{lemma}

\begin{proof} This is proved analogously to Lemma~\ref{le_1bdd}.\end{proof}

\begin{lemma}[Reduction to pretentious case]\label{le2} Suppose that the completely multiplicative function $f:\mathbb{N}\to \mathbb{C}$ has property R. Then $\mathbb{D}(f,1:\infty)<\infty$.
\end{lemma}

\begin{proof}
Since we have $|\sum_{n\leq x}f(n)|\gg x$ and also $|f(n)|\leq 1$ by Lemma \ref{le0}, from Lemma \ref{le_hal} we deduce the existence of some real number $t$ such that
\begin{align}\label{eq5}
\sum_{p}\frac{1-\Re(f(p)p^{-it})}{p}<\infty.    
\end{align}
We suppose for the sake of contradiction that $t\neq 0$. Lemma \ref{le_hal} then gives the asymptotic
\begin{align*}
\frac{1}{x}\sum_{n\leq x}f(n)=\frac{x^{it}}{1+it}\prod_{p}\Big(1-\frac{1}{p}\Big)\Big(1+\frac{f(p)}{p^{1+it}}+\frac{f(p)^2}{p^{2+2it}}+\cdots\Big)+o(1):=c'x^{it}+o(1)   
\end{align*}
for some constant $c'$, since the product over $p$ above converges by \eqref{eq5}. But since $f$ satisfies property R with some constant $c$, this implies
\begin{align*}
c=c'x^{it}+o(1),    
\end{align*}
as $x\to \infty$, which is an evident contradiction since $m\mapsto m^{it}$, $m\geq x_0$ is dense on the unit circle. Thus \eqref{eq5} holds with $t = 0$.
\end{proof}

\begin{proof}[Proof of Theorem~\ref{theo_ruzsa}]
Applying Lemma~\ref{le0} and Lemma~\ref{le2}, we may assume that $\mathbb{D}(f,1;\infty)<\infty.$ Let $\varepsilon, H,w$ be parameters obeying \eqref{equ6}.

Let $W=\prod_{p\leq w}p^{w}$ for some $w\ge 1.$ Assuming that $f(p)\ne 1$ for infinitely many primes $p,$ we can repeat the argument as in the proof of Proposition \ref{all_finite} verbatim 
with $\chi$ being trivial character and $t=0$ to end up with a contradiction.

Consequently, we let $p_1,\dots, p_k$ be the primes for which $f(p_i)\neq 1.$ If $|f(p_i)|<1$ for all $i\le k$ we are done by the if direction of the proof. Otherwise, choose some $r$ such that $f(r)\neq 1$ and $|f(r)|=1$, and let $z=f(r)\neq 1$. Let $Q$ be the product of all $p_i$, $p_i\neq r$. By Lemma \ref{le_perturb}, it suffices to prove that if $\chi^{*}\pmod Q$ is the principal character, then the function $\chi_{r,z}^{*}$ (defined as in \eqref{equ8}) does not satisfy
\begin{align}\label{equ9}
\sum_{n\leq x}\chi_{r,z}^{*}(n)=c'x+O(1)    
\end{align}
for any $c'\neq 0$. If we define $Q'=rQ$ and $R(y):=\sum_{n\leq y}1_{(n,Q')=1}-(\varphi(Q')/Q')y$, then the left-hand side of \eqref{equ9} can be written as
\begin{align*}
\sum_{k\leq 2\log x}z^k\Big(\frac{\varphi(Q')}{Q'}\cdot \frac{x}{r^k}+R\Big(\frac{x}{r^k}\Big)\Big) =\frac{\varphi(Q')}{Q'}\cdot \frac{1}{1-z/r}x+ \sum_{k\geq 0}z^kR\Big(\frac{x}{r^k}\Big)+O(1).  
\end{align*}
The last sum over $k$ is clearly $O(\log x)$, so we must have $c'=(\varphi(Q')/Q')(1-z/r)^{-1}$ and
\begin{align*}
 \sum_{k\geq 0}z^kR\Big(\frac{x}{r^k}\Big)=O(1).   
\end{align*}
However, following the proof of Lemma \ref{zero_sum} verbatim, this can happen only if
\begin{align*}
R(mQ)=0\quad \textnormal{for all}\quad m\geq 1.    
\end{align*}
But then following the proof of Lemma \ref{zero_sumfinal} verbatim, we conclude that $z=\chi^{*}(r)=1$, which is a contradiction, and this finishes the proof.\end{proof}

\begin{proof}[Proof of Corollary \ref{cor_glaz} a)]
Let $f: \mb{N} \ra \mb{C}$ be a completely multiplicative function which satisfies \eqref{eq1} and  $|f(n)|\in \{0,1\}$ for all $n$. By Theorem \ref{theo_ruzsa} there is a finite set $S$ of primes such that for all $p \notin S$ we have $f(p) = 1$, and otherwise $|f(p)|<1$, so actually $f(p)=0$. But this means that $f(n)=1_{(n,Q)=1}$ for some $Q\geq 1$.\end{proof}

\section{Proof of Theorem \ref{thm_sfedp}}\label{others}

In this section, we apply the rotation trick to prove the square-free discrepancy conjecture (Theorem \ref{thm_sfedp}). Unlike in the previous sections, we need to apply the trick to functions that are not completely multiplicative.  The proof is carried out in two steps, the first of which is the following proposition that reduces us to the case of modified characters.

\begin{prop}\label{prop_redaym}
Let $g: \mb{N} \ra \{-1,+1\}$ be a multiplicative function, and let $f = \mu^2 g$. Assume that $f$ pretends to be a real quadratic character $\chi$ modulo $q$, and let $S := \{p : f(p) \neq \chi(p)\}$. If 
$$
\limsup_{x \ra \infty} \Big|\sum_{n \leq x} f(n) \Big| < \infty,
$$
then $|S| < \infty$.
\end{prop}

\begin{proof}
Note that $f(n) = \mu^2(n) \prod_{p|n} g(p)$, so we may assume without loss of generality that $g$ is completely multiplicative. 

Let $H$ be sufficiently large in terms of $q$, and for $m \in \mb{N}$ set $I_H(m) := [(H!)^2m+1,(H!)^2m+H]$. Let $1 \leq \ell \leq H$. Fix a set $\mbf{r} = \{r_1,\ldots,r_{\ell}\} \subset [1,H]$ such that 
\begin{enumerate}[(i)]
    \item $\mu^2(r_j) = 1$ for all $1 \leq j \leq \ell$;
    \item $p|r_j \Rightarrow p \notin S$, for all $1 \leq j \leq \ell$;
\end{enumerate}
the existence of such tuples, for $H$ large enough in terms of $\ell$, follows from a lower bound sieve. By the pigeonhole principle we can then select a subset $\mbf{r}'=\{r_1',\ldots,r_t'\} \subseteq \mbf{r}$ such that $g(r_j') = g(r_k')$ for all $1 \leq j,k \leq t$, with $t \geq \ell/2$. 

Assume that $|S| = \infty$. Then we can pick distinct primes $p_1,\ldots,p_t > H$ with $p_j \in S$ for all $1 \leq j \leq t$. With the $t$-tuple $\mbf{r}'$ chosen above and $x$ chosen large as a function of $H$, we define the set
\begin{align*}
\mc{M}_{\mbf{r}'}(x) := \Big\{m \leq x : \Big( p \in S\, \wedge\, p|\prod_{n \in I_H(m)} n \Rightarrow p \leq H\Big) \text{ and } \mu^2((H!)^2m+r) =1_{r\in \mbf{r}'}, r \in [1,H]\Big\}.
\end{align*}
We claim that $\mc{M}_{\mbf{r}'}(x) \gg_H x$, which we may prove as follows.

Let $w$ be large enough in terms of $H$, and $H$ large enough relative to $\ell$ in order for the following estimates to hold:
$$
H\sum_{\substack{p > w \\ p \in  S}} \frac{1}{p}, \ell \sum_{p > w} \frac{1}{p^2} \ll 1/H
$$
For each $j \notin \mbf{r}'$ fix a prime $P_j \in (w,10 w] \cap S^c$, and define $P \leq (10w)^{H} = O_H(1)$ to be their product. By the Chinese remainder theorem, there is a residue class $a \pmod{P^2}$ such that if $m \equiv a \pmod{P^2}$ then $(H!)^2m + j \equiv 0 \pmod{P_j^2}$ for all $j \notin \mbf{r}'$. 

Set $W := \prod_{p \leq w} p$ and consider the arithmetic progression 
$$
\mc{A}(x) := \{m \leq x : m \equiv a \pmod{P^2}, m \equiv 0 \pmod{W^2}\} := \{m \leq x : m \equiv a' \pmod{(PW)^2}\}.
$$ 
We claim that $|\mc{A}(x) \bk \mc{M}_{\mbf{r}'}(x)| \leq |\mc{A}(x)|/2$ for $H$ large enough, which implies in particular that 
$$
|\mc{M}_{\mbf{r}'}(x)| \gg |\mc{A}| \gg x/(WP)^2 \gg_H x,
$$
as required. To see this, let $m \in \mc{A}(x) \bk \mc{M}_{\mbf{r}'}(x)$. Thus, either: (a) there is a $1 \leq j \leq H$ for which $p | ((H!)^2m + j)$ for some $p \in S$, $p > H$, or (b) there is $r_j' \in \mbf{r}'$ such that $p^2|((H!)^2m + r_j')$ for some prime $p$. 

In case (a), as $W^2| m$ it follows in fact that $p > w$, and as $p \in S$ we have that $p \nmid PW$. Thus, by the union bound and the Chinese remainder theorem the number of $m$ in question is 
$$
\ll \sum_{1 \leq j \leq H} \sum_{\substack{p > w \\ p \in S}} \sum_{\substack{m \leq X \\ m \equiv a' \pmod{(PW)^2} \\ (H!)^2 m + j \equiv 0 \pmod{p}}} 1 \ll \frac{X}{(PW)^2} \cdot H\sum_{\substack{p > w \\ p \in S}} \frac{1}{p} \ll \frac{X}{(PW)^2 H}.
$$
In case (b), if $p^2|((H!)^2m+r_j')$ then, again as $W^2|m$ we have $p > w$ (with $p \nmid P$ since otherwise $p = p_i$ for some $i$, whence $p|(r_j'-i)$ a contradiction to the fact that $p > w \geq H$). Thus, again by the union bound and the Chinese remainder theorem the number of such elements of $\mc{A}(x)$ is bounded by
$$
\leq \sum_{1 \leq j \leq t} \sum_{\substack{p > w \\ p \nmid P}} \sum_{\substack{m \leq X \\ m \equiv a' \pmod{(PW)^2} \\ (H!)^2m + r_j'\equiv 0 \pmod{p^2}}} 1 \ll \frac{X}{(PW)^2} \cdot \ell \sum_{ p > w} \frac{1}{p^2} \ll \frac{X}{(PW)^2 H}.
$$
In summary, we obtain
$$
|\mc{A}(x) \backslash \mc{M}_{\mbf{r}'}(x)| \ll \frac{X}{(PW)^2H} \ll |\mc{A}(x)|/H,
$$
which is indeed $\leq |\mc{A}(x)|/2$ as soon as $H$ is sufficiently large, as required.

We also define $\mc{N}_{\mbf{r}'}(x)$ to be the set of $m \leq x$ satisfying the following properties: 
\begin{enumerate}[(i)]
    \item  if $p \in S$ and $p|\prod_{\substack{1 \leq r \leq H \\ r \not \in \mbf{r}'}} ((H!)^2m+r)$, then $p \leq H$;
\item  for each $1 \leq j \leq t$, $(H!)^2 m + r_j' \equiv p_j \pmod{p_j^2}$, and if $p \in S \backslash \{p_j\}$ then $p \nmid ((H!)^2m+r_j)$;

\item We have $\mu^2((H!)^2m+r) = 1_{r\in \mbf{r}'}$.\\
By a similar argument as for $\mc{M}_{\mbf{r}'}(x)$, this set satisfies $\mc{N}_{\mbf{r}'}(x) \gg_{p_j,H} x$ by the Chinese remainder theorem and a lower bound sieve.
\end{enumerate}

We now pick $m' \in \mc{N}_{\mbf{r}'}(x)$ and $m \in \mc{M}_{\mbf{r}'}(x)$ with $m' > m$ and note that (using notation from Subsection \ref{subsec_notation}) $((H!)^2m+r_j',S) = p_j(r_j', S \cap [1,H]) = p_j$. Now, if we set 
$$
P_r(m) := \prod_{\substack{p^k|| (H!)^2m+r \\ p \in S}} p^k,
$$ 
then for $m \in \mc{M}_{\mbf{r}'}(x)$ we have $P_{r_j'}(m)=r_j'/\prod_{p^k\mid \mid  r_j', p\not \in S}p^k$ and recalling $g(p)=\chi(p)$ for $p\not\in S$ this implies
$$
g((H!)^2m+r_j') = g(P_{r_j'}(m)) \chi\Big(\frac{(H!)^2}{P_{r_j'}(m)} m + \frac{r_j'}{P_{r_j'}(m)}\Big) = g(P_{r_j'}(m)) \chi\Big(\frac{r_j'}{P_{r_j'}(m)}\Big) = g(r_j'),
$$ 
for all $1 \leq j \leq t$, where the last equality comes from  $g(r_j)=\prod_{p^k\mid \mid r_j, p\in S}g(p^k)\prod_{p^k\mid \mid r_j, p\not \in S}g(p^k)$. On the other hand, as $((H!)^2m'+r_j', S) = p_j \cdot (r_j',S\cap [1,H])$ for $m'\in \mc{N}_{\mbf{r}'}(x)$, we have $P_{r_j'}(m') = p_jP_{r_j'}(m)= p_jr_j'/\prod_{p^k\mid \mid  r_j', p\not \in S}p^k$, and so
\begin{align*}
g((H!)^2m'+r_j') &= g(p_j) g(P_{r_j'}(m)) \chi\Big(\frac{(H!)^2 m' + r_j'}{p_jP_{r_j'}(m)}\Big) = g\chi(p_j) g(P_{r_j'}(m)) \chi\Big(\frac{(H!)^2}{P_{r_j'}(m)}m + \frac{r_j'}{P_{r_j'}(m)}\Big) \\
&= g\chi(p_j) g(P_{r_j'}(m)) \chi\Big(\frac{r_j'}{P_{r_j'}(m)}\Big) = g\chi(p_j) g(r_j').
\end{align*}
Using the preceding two formulas and recalling the definitions of $\mathcal{M}_{\mbf{r}'}(x)$ and $\mathcal{N}_{\mbf{r}'}(x)$, it follows that
\begin{align*}
\Big|\sum_{n \in I_H(m')} f(n) - \sum_{n \in I_H(m)} f(n)\Big| &= \Big|\sum_{\substack{n \in I_H(m') \\ \mu^2(n) = 1}} g(n) - \sum_{\substack{n \in I_H(m) \\ \mu^2(n) = 1}} g(n)\Big| \\
&= \Big|\sum_{1 \leq j \leq t} (g((H!)^2 m' + r_j')-g((H!)^2m+r_j'))\Big| \\
&= \Big|\sum_{1 \leq j \leq t} (1-(g\chi(p_j)))g(r_j')\Big| = 2t \geq \ell
\end{align*}
given that $g(r_j)$ has the same sign for all $1 \leq j \leq t$, and $g\chi(p_j) = -1$ on $S$. Now since $H$ can be taken large, $\ell$ can also be chosen as large as desired; on the other hand, we have
\begin{align}\label{equ7}
\Big|\sum_{n \in I_H(m')} f(n) - \sum_{n \in I_H(m)} f(n)\Big| \leq 4 \sup_{1 \leq y \leq 2(H!)^2x} \Big|\sum_{n \leq y} f(n)\Big|.
\end{align}
Thus the right-hand side of \eqref{equ7} can be made arbitrarily large. This is a contradiction, so the claim $|S| < \infty$ follows.
\end{proof}

The case of modified characters that remains is still nontrivial, and we handle it in the following proposition.

\begin{prop}[A Dirichlet series argument]\label{prop_dirichlet} Let $g:\mathbb{N}\to \{-1,+1\}$ be multiplicative, and let $f=\mu^2g$. Suppose that there exists a real character $\chi$ such that $g(p)=\chi(p)$ for all but finitely many $p$. Then for any $\varepsilon>0$ we have
\begin{align*}
\sum_{n\leq x}f(n)=\Omega(x^{1/4-\varepsilon}).    
\end{align*}
\end{prop}

\begin{remark}
A version of this result conditional on GRH was obtained by Aymone \cite[Thm 1.3]{aym2}.
\end{remark}

\begin{proof} We may assume that $\chi\pmod q$ is non-principal, as otherwise a bound of $\gg x$ for the partial sums follows from Delange's theorem (Lemma \ref{le_hal}). 

Consider the Dirichlet series $F(s) := \sum_{n \geq 1} f(n)/n^s$ of $f$, which defines an analytic function in the region $\text{Re}(s) > 1$. Also, put $M_f(x) := \sum_{n \leq x} f(n)$. We will show that $|M_f(x)|=O(x^{1/4-\varepsilon})$ cannot hold for any fixed $\varepsilon>0$. 

By partial summation, we have
$$
F(s) = \int_1^{\infty} M_f(x) x^{-s-1} dx,
$$
initially for $\Re(s)>1$, but if $M_f(x)=O(x^{1/4-\varepsilon})$, then this formula extends $F$ analytically to the half-plane $\text{Re}(s) > 1/4-\varepsilon$. 

Let $q'$ be the product of $q$ and those finitely many primes for which $g(p)\neq \chi(p)$. Now, comparing Euler products when $\text{Re}(s) > 1$, we note that since $\chi(p)^2 = 1$ for all $p\nmid q'$,
\begin{align*}
F(s) &= \prod_p \Big(1+g(p)p^{-s}\Big) = \prod_{p|q'} (1+g(p)p^{-s}) \prod_{p\nmid q'} (1+\chi(p)p^{-s}) \\
&= \prod_{p|q'} (1+g(p)p^{-s}) \prod_{p \nmid q'} (1-p^{-2s})(1-\chi(p)p^{-s})^{-1} \\
&= \prod_{p|q'} (1+g(p)p^{-s})(1-\chi(p)p^{-s})(1-p^{-2s})^{-1} L(s,\chi)\zeta(2s)^{-1}\\
&=: P(s) L(s,\chi)\zeta(2s)^{-1}.
\end{align*}
Note that $P(s)$ extends analytically to $\text{Re}(s) > 0$ and has no zeros there, so by analytic continuation we see that $L(s,\chi)\zeta(2s)^{-1} = F(s)/P(s)$ is analytic in the half-plane $\text{Re}(s) > 1/4-\varepsilon$. 

By an old result of Hardy and Littlewood \cite{har_lit}, we can find $\gg T$ real numbers $\gamma \in [-T,T]$ such that $\zeta(1/2+2i\gamma) = 0$. Thus, there are $\gg T$ points $\rho = 1/4 + i\gamma$ at which $\zeta(2\rho) = 0$. Since $L(s,\chi)/\zeta(2s)$ must be analytic along $\text{Re}(s) = 1/4$, we have $L(\rho,\chi) = 0$ for these $\gg T$ points $\rho$, and thus 
\begin{align}\label{eq15}
|\{(\sg,t) \in (0,1/4]\times [-T,T]: L(\sg + it,\chi) = 0\}| \gg T.
\end{align}
On the other hand, since $\chi$ is real, the functional equation gives that $L(\sg + it,\chi) = 0$ if and only if $L(1-\sg - it,\chi) = 0$ whenever $\sg \in (0,1)$, and combining this with Huxley's zero density estimate \cite{huxley}, we have
\begin{align*}
|\{(\sg,t) \in (0,1/4]\times [-T,T]: L(\sg + it,\chi) = 0\}| &= |\{(\sg,t) \in [3/4,1) \times [-T,T] : L(\sg + it,\chi) = 0\}| \\
&\ll_{\eta} (qT)^{(12/5+\eta)(1-3/4)} \ll T^{0.7} = o(T),
\end{align*}
which contradicts \eqref{eq15}. (Note that we only needed a very weak saving here, and thus also some zero density estimates older than Huxley's would have sufficed.) 

It follows that $F(s)$ is not analytic in $\text{Re}(s) > 1/4-\varepsilon$ if $\varepsilon > 0$, and therefore $|M_f(x)|\gg x^{1/4-\varepsilon}$ for infinitely many integers $x\geq 1$, which is the desired contradiction.
\end{proof}

\begin{proof}[Proof of Theorem \ref{thm_sfedp}]
Assume for the sake of contradiction that there is a multiplicative function $g : \mb{N} \ra \{-1,+1\}$ such that for $f := \mu^2 g$ we have
\begin{align}\label{eq16}
\limsup_{x \ra \infty} \Big|\sum_{n \leq x} f(n) \Big| < \infty.
\end{align}
By \cite[Theorem 1.1]{Ayo}, there is a real, primitive Dirichlet character $\chi$ modulo $q$, where $(q,6) = 1$, such that $f$ pretends to be $\chi$, i.e., $\sum_p (1-f(p)\chi(p))/p < \infty$, and such that $f(2)\chi(2) = f(3)\chi(3) = -1$. Now, by Proposition \ref{prop_redaym}, the set $S := \{p : f(p)\chi(p) = -1\}$ must be finite. But then $g(p)=\chi(p)$ for all but finitely many primes, in which case Proposition \ref{prop_dirichlet} gives the desired a contradiction.\end{proof}

\section{Examples where \texorpdfstring{$f(p)$}{f(p)} vanishes at many primes} \label{sec_exs}
\noindent We now construct various types of  completely multiplicative functions $f$ that satisfy the assumptions of Theorem~\ref{thm_bddsums} apart from $\sum_p (1-|f(p)|)/p<\infty$, and that look nothing like twisted Dirichlet characters. In particular, we will see that there are $\mathfrak{c}=|\mathbb{R}|$ such functions $f:\mathbb{N}\to \mathbb{Q}\cap [-1,1]$.

Note that any function $f$ for which 
\begin{align}\label{eq30}
\sum_{n\leq x}|f(n)|=O(1) 
\end{align}
trivially has bounded partial sums. We call functions not satisfying \eqref{eq30} \emph{nontrivial}.

i) Let $p$ be a prime. Let $\alpha \in (0,1)$, and define a completely multiplicative function by $f(p) = e^{2\pi i \alpha}$, $f(p') = 0$ for all $p' \neq p$. Then $f$ is nontrivial. We clearly have
\begin{align*}
\sum_{n \leq x} f(n) = \sum_{\ell \leq \log x/\log p} e(\alpha \ell) = O_{\alpha}(1),
\end{align*}
using the geometric sum formula.

ii) Let $\mathbf{h}:\mathbb{N}\to \{-1,+1\}$ be a random completely multiplicative multiplicative function (the values at the primes being independent Rademacher random variables). Let $f(n)=\mathbf{h}(n)/n^{r}$ for any fixed $r\in [0,1]$. Then $f$ is nontrivial. Moreover, if $r>1/2$, almost surely we have 
\begin{align}\label{random}
\sum_{n\leq x}f(n)=O(1),    
\end{align}
whereas if $r< 1/2$, then almost surely \eqref{random} fails. Thus there is a transition in behavior at $r=1/2$. The above implies that if we remove from Theorem \ref{thm_bddsums} the assumption \eqref{eq_thin}, then there are cardinality $\mathfrak{c}=|\mathbb{R}|$ counterexamples, even when restricted to functions $f:\mathbb{N}\to \mathbb{Q}\cap [-1,1]$.

iii) When it comes to deterministic versions of the example above, if $f(n)=\lambda(n)/n$ where $\lambda$ is the Liouville function, then by the prime number theorem and partial summation (and the fact that $\sum_{n\geq 1} \lambda(n)/n=0$), for $x$ sufficiently large we have
$$
\Big|\sum_{n \leq x}f(n)\Big| \ll_A (\log x)^{-A},
$$
which establishes rapid decay of the partial sums towards zero in this case. More generally, if $f(n)=h(n)/n$ with $\mathbb{D}(h,\lambda;\infty)<\infty$, then $h$ is nontrivial and the partial sums of $f$ converge to zero by \cite[Theorem 2.4]{goldmakher}. Thus, the partial sums of a nontrivial $f$ can even converge to zero without $f$ resembling a character.

iv) By Lemma \ref{le_perturb}, if we perturb any of the above functions $f$ by choosing a finite set $S$ of primes and changing their values at $p\in S$ to be any numbers $z_p$ with $|z_p|<1$, then the partial sums continue to be bounded.

As noted before, the assumption \eqref{eq_thin} in Theorem \ref{thm_bddsums} is equivalent to $\sum_{n\leq x}|f(n)|=\Omega(x)$. Let us indicate why relaxing this assumption even a bit, say to $\sum_{n\leq x}|f(n)|=\Omega(x/(\log x)^{\delta})$, would either require a completely different approach or considerable progress on correlations of multiplicative functions. If we follow the reduction to the pretentious case in Lemma \ref{le_pret} for such ``sparse'' $f$, then to achieve that reduction we would need
\begin{align*}
\max_{1\leq |h|\leq H}\sum_{n\leq x}\frac{f(n)\bar{f}(n+h)}{n}=o\Big(\frac{1}{H}\sum_{n\leq x}\frac{|f(n)|^2}{n}\Big)   
\end{align*}
for $H=(\log x)^{\delta}$ (we cannot take $H$ smaller than this, since otherwise we would expect the right-hand side of \eqref{equ10} to be $O(1)$). Since $f$ is a sparse function, no such estimate is currently known, even for fixed $h$. Also the range of uniformity $1\leq h\leq H$ required here goes beyond what is currently known even for the case of non-sparse functions $f:\mathbb{N}\to \mathbb{U}$. Therefore, obtaining such a relaxation of our assumption appears challenging.

\bibliography{ruzsarefs}
\bibliographystyle{plain}

\end{document}